\documentclass{cimart}

\usepackage[capitalise]{cleveref}

\DeclareMathOperator{\id}{id}

\title{
    Combinatorial twists in $\mathfrak{gl}_n$ Yangians
    }

\author{
    Anastasia Doikou
    }

\authorinfo{Heriot-Watt University and Maxwell Institute for Mathematical Sciences, UK}{a.doikou@hw.ac.uk}

\abstract{%
    We introduce the special set-theoretic Yang-Baxter algebra and show that it is a Hopf algebra subject to certain conditions. The associated universal ${\cal R}$-matrix is also obtained via an admissible Drinfel'd twist. The structure of braces emerges naturally in this context by requiring the special set-theoretic Yang-Baxter algebra to be a Hopf algebra and a quasi-triangular bialgebra after twisting. The fundamental representation of the universal ${\cal R}$-matrix yields the familiar set-theoretic (combinatorial) solutions of the Yang-Baxter equation. We then apply the same Drinfel'd twist to the $\mathfrak{gl}_n$ Yangian after introducing the {\it augmented Yangian}. We show that the augmented Yangian is also a Hopf algebra and we also obtain its twisted version.
    }



\VOLUME{33}
\YEAR{2025}
\ISSUE{3}
\NUMBER{12}
\DOI{https://doi.org/10.46298/cm.15770}

\begin{document}

\section{Introduction} 

The main aim of this study is the use of certain universal Drinfel'd twists \cite{Drinfeld, Drinfeld2} in the context of $\mathfrak{gl}_n$ Yangians ${\cal Y}(\mathfrak{gl}_n).$  We focus on universal twists that are combinatorial matrices in the fundamental representation \cite{Doikoutw, DoRySt} and generate combinatorial (set-theoretic) solutions of the Yang-Baxter equation (see, for instance, \cite{Dr92, Eti, Sol, GatMaj, Gat1, Gat2, JeOk, Rump1, Rump2}). In this manuscript the solutions of the Yang-Baxter equation are expressed as $n^2 \times n^2$ matrices \cite{DoiSmo1, DoiSmo2}.  In this spirit, when we say combinatorial solutions we mean that the matrices that represent the solutions of the Yang-Baxter equation are combinatorial, i.e. they have only one nonzero element, which takes the value 1, in every row and column.
We consider the linearized version of the set-theoretic Yang-Baxter equation and derive 
the quasi-triangular bialgebras associated to set-theoretic solutions. By identifying a suitable admissible Drinfel’d twist we are able to extract the
general set-theoretic universal ${\cal R}$-matrix, which is an element of ${\cal A} \otimes {\cal A}$ and ${\cal A}$ is the underlying bialgebra. 

More specifically, 
in Section \ref{2} we introduce the special set-theoretic Yang-Baxter algebra and show that it is a Hopf algebra. The associated universal ${\cal R}$-matrix is also obtained via an admissible Drinfel'd twist, making the special set-theoretic Yang-Baxter algebra a quasi-triangular bialgebra. The fundamental representation of the universal ${\cal R}$-matrix gives typical set-theoretic (combinatorial) solutions of the Yang-Baxter equation.
Here we only obtain reversible universal ${\cal R}$ matrices, i.e. ${\cal R}_{12}{\cal R}_{21} =1_{\cal A \otimes {\cal A}}$ (and their representations) as opposed to the general scenario discussed in \cite{DoRySt}, where rack type solutions of the Yang-Baxter equation and their Drinfel'd twists were discussed. The key novel outcomes of Section \ref{2} are summarized in Theorems \ref{basica2b}, \ref{corf0} and \ref{twist2b}, where we show that the algebraic structure of (skew) braces (Definition \ref{defbrace}) \cite{Rump1, Rump2, GV}, emerge naturally if the special set-theoretic Yang-Baxter algebra is required to be a Hopf algebra and a quasi-triangular bialgebra after twisting. In fact, it turns out that the twisted Hopf algebra is a quasi-triangular Hopf algebra (Theorem \ref{twist2b}). 
The more general new results of the present investigation are presented in Section \ref{3}, where we extend our analysis to the $\mathfrak{gl}_n$ Yangian and to parametric solutions of the Yang-Baxter equation. Specifically, we introduce the augmented $\mathfrak{gl}_n$ Yangian, we show in Theorem \ref{bob} that it is Hopf algebra and using the set-theoretic Drinfel'd twist we are able to obtain its twisted version. We basically extend the results of \cite{Doikoutw, DoGhVl}, where only fundamental representations of the augmented Yangian and the twisted $R$-matrix were presented. 

Before we continue with our analysis and the presentation of the main results, we recall the basic definitions of Hopf and quasi-triangular Hopf algebras, which will be used later in our analysis.

We first recall the definition of the Hopf algebra (see, for instance, \cite{Chari, Majid})
\begin{definition} \label{hopf}
A Hopf algebra $({\cal A}, \Delta, \epsilon, s)$ is a unital, associative algebra ${\cal A}$ over some field $k$ equipped  with the following linear maps: 
\begin{itemize}
\item multiplication, $m: {\cal A} \times {\cal A} \to {\cal A},$ $m(a,b) = ab,$ 
which is associative $(ab)c = a (bc)$ for all $a,b,c \in {\cal A}$
\item $\eta: k \to {\cal A},$ such that it produces the unit element for ${\cal A},$ $\eta(1) = 1_{\cal A}.$
\item co-product, $\Delta: {\cal A} \to {\cal A} \otimes {\cal A},$ $\Delta(a) = \sum_j \alpha_j \otimes \beta_j,$ which is coassociative, \[(\id \otimes \Delta)\Delta(a) = (\Delta \otimes \id)\Delta(a), \quad \text{ for all } a\in {\cal A}.\]
\item co-unit, $\epsilon: {\cal A} \to k,$ such that $(\epsilon \otimes \id)\Delta(a) =  (\id \otimes \epsilon)\Delta(a)= a,$ for all $a \in {\cal A}.$ 
\item antipode, $s: {\cal A} \to {\cal A},$ (bijective map) such that \[ m (s \otimes \id)\Delta(a) = m(\id \otimes s) \Delta(a) = \epsilon(a)  1_{\cal A}, \quad \text{for all } a \in {\cal A}.\]
\item $\Delta, \epsilon$ are algebra homomorphisms and ${\cal A} \otimes {\cal A}$ has the structure of a tensor product algebra: $(a \otimes b)(c \otimes d) = ac\otimes bd,$ for all $a,b,c,d \in {\cal A}$.
\end{itemize}
\end{definition}
If we do not require the existence of an antipode then $({\cal A}, \Delta, \epsilon)$ is called a {\it bialgebra}.

We also recall the definition of a quasi-triangular Hopf algebra \cite{Drinfeld, Drinfeld2}.
\begin{definition} \label{quasi}
Let ${\cal A}$ be a Hopf algebra over some field $k$, then ${\cal A}$ is a quasi-triangular Hopf algebra if there exists an invertible element ${\cal R} \in {\cal A}\otimes {\cal A}$ (universal ${\cal R}$-matrix), such that:
\begin{enumerate}
    \item  ${\cal R} \Delta(a) =\Delta^{op}(a) {\cal R},$ for all $a \in {\cal A},$ 
    where $\Delta : {\cal A} \to {\cal A} \otimes {\cal A}$ is the co-product on ${\cal A}$ and $\Delta^{op}(a) = \pi \circ \Delta(a),$ $\pi : {\cal A} \otimes {\cal A} \to {\cal A} \otimes {\cal A},$ such that $\pi (a \otimes b) = b \otimes a.$ 
    \item $(\id \otimes \Delta){\cal R} = {\cal R}_{13} {\cal R}_{12}$, and $(\Delta \otimes \id){\cal R} = {\cal R}_{13} {\cal R}_{23}.$
\end{enumerate}
\end{definition}
Also, the following statements hold:
\begin{itemize}
\item The antipode 
$s: {\cal A} \to {\cal A}$ 
satisfies $(\id \otimes s){\cal R}^{-1} = {\cal R},$ 
$(s \otimes \id){\cal R} ={\cal R}^{-1}.$
\item The co-unit $\epsilon: {\cal A} \to k$ satisfies $(\id \otimes \epsilon) {\cal R} = (\epsilon \otimes \id){\cal R} = 1_{\cal A}.$
\item Due to Definition \ref{quasi} the universal ${\cal R}$-matrix satisfies the Yang-Baxter equation  
\begin{equation}
{\cal R}_{12} {\cal R}_{13} {\cal R}_{23} ={\cal R}_{23} {\cal R}_{13} {\cal R}_{12}. \label{UYBE}
\end{equation}
We recall the index notation: let
${\cal R} = \sum_j a_j \otimes b_j,$ then ${\cal R}_{12} = \sum_j a_j \otimes b_j \otimes 1_{\cal A}$, ${\cal R}_{23} = \sum_j 1_{\cal A} \otimes a_j \otimes b_j$ and ${\cal R}_{13} = \sum_j a_j \otimes 1_{\cal A} \otimes  b_j.$\end{itemize}
Proofs of the above statements can be found, for instance, in \cite{Chari, Majid}.

\begin{remark} \label{rem00}
Consider a representation $\rho_{\lambda}: {\cal A} \to \mbox{End}({\mathbb C}^n),$ $\lambda \in {\mathbb C},$ such that
\[(\rho_{\lambda} \otimes \id){\cal R} =: L(\lambda) \in \mbox{End}({\mathbb C}^n) \otimes {\cal A},\] and
$(\rho_{\lambda_1} \otimes \rho_{\lambda_2}){\cal R} =: R(\lambda_1, \lambda_2)\in \mbox{End}({\mathbb C}^n) \otimes \mbox{End}({\mathbb C}^n),$ $\lambda_{1,2} \in {\mathbb C}.$ Then the Yang-Baxter equation (\ref{UYBE}) reduces to (we suppress the index $3$ in the following equation)
\begin{equation}
R_{12}(\lambda_1, \lambda_2)L_1(\lambda_1) L_2(\lambda_2) =  L_2(\lambda_2)L_1(\lambda_1) R_{12}(\lambda_1, \lambda_2) \label{rll}
\end{equation}
after acting with $(\rho_{\lambda_1} \otimes \rho_{\lambda_2} \otimes \id)$ on (\ref{UYBE}).
Moreover, (\ref{UYBE}) turns into the Yang-Baxter equation on $({\mathbb C}^n)^{\otimes 3},$ 
\begin{equation}
R_{12}(\lambda_1, \lambda_2)R_{13}(\lambda_1, \lambda_3) R_{23}(\lambda_2,\lambda_3) =  R_{23}(\lambda_2,\lambda_3)R_{13}(\lambda_1,\lambda_3) R_{12}(\lambda_1, \lambda_2)
\nonumber
\end{equation}
after acting with $(\rho_{\lambda_1} \otimes \rho_{\lambda_2} \otimes \rho_{\lambda_3}).$ An equation similar to (\ref{rll}) holds for 
\[\hat L(\lambda) := (\id \otimes \rho_{\lambda}) {\cal R} \in {\cal A} \otimes \mbox{End}({\mathbb C}^n),\] and a mixed equation for both $L,\ \hat L$ also follows from (\ref{UYBE}). 
\end{remark}

\section{Set-theoretic Hopf algebras} \label{2}

In this section, we introduce the {\it special set-theoretic Yang-Baxter algebra} (or special set-theoretic YB algebra for the sake of brevity) and show that it is a Hopf algebra. Then by introducing a suitable Drinfel'd twist we derive the universal ${\cal R}$-matrix associated to the special set-theoretic YB algebra (see a general analysis and definitions in \cite{Doikoutw, DoRySt}, see also \cite{Sol, LebVen}). 
We also show that the special set-theoretic Hopf algebra becomes a quasi-triangular bialgebra after twisting, subject to certain conditions that naturally lead to the structure of (skew) braces.

\subsection{Special set-theoretic YB algebra as a Hopf algebra}
We first define the basic set-theoretic YB algebra as follows (see also \cite{DoRySt}).
\begin{definition}  
\label{setalgd1} 
Let $X$ be a non-empty set, and for all $a,b\in X,$ $\sigma_a, \ \tau_b: X\to X.$  
We say that the unital, associative algebra ${\cal A}$ over $k,$
generated by indeterminates $1_{{\cal A}}$ (unit element), $h_a,$
$w_a, w^{-1}_a,$ for $a \in X,$ 
and relations for all $a,b \in X:$
\begin{eqnarray}
 h_a  h_b =\delta_{a,b} h_a, ~~ w_a^{-1} w_a =w_aw_a^{-1} =1_{ {\cal A}}, ~~  w_a w_b = 
w_{\sigma_a(b)} w_{\tau_{b}(a)} ~~  w_a h_b = h_{\sigma_a(b)} w_a,  \label{qualgbb}
\end{eqnarray}
is a {\it basic set-theoretic YB algebra}.
\end{definition}

\begin{proposition} 
\label{qua2} 
Let ${\cal A}$ be the basic set-theoretic YB algebra
then 
\begin{equation}
h_{\sigma_{\sigma_{a}(b)}(\sigma_{\tau_{b}(a)}(c))} = h_{\sigma_a(\sigma_b(c))}.  
\label{basiko0} \nonumber
\end{equation}
If, in addition, for all $a,b \in X,$ $h_a = h_b \Rightarrow a =b,$ then for all $a,b,c \in X,$
\begin{eqnarray}
&&  \sigma_{a}(\sigma_b(c))= \sigma_{\sigma_{a}(b)}(\sigma_{\tau_{b}(a)}(c)) \label{basic00}
\end{eqnarray}
and $\sigma_a: X \to X$ is an injection.
\end{proposition}
\begin{proof}
We compute $w_a w_b h_c$ using the associativity of the algebra, relations (\ref{qualgbb}) and the invertibility of $w_a,$ for all $a\in X$ we conclude for all $a,b,c \in X$ (see also \cite{DoRySt}) \begin{equation}
h_{\sigma_{\sigma_{a}(b)}(\sigma_{\tau_{b}(a)}(c))} = h_{\sigma_a(\sigma_b(c))}\  
\Rightarrow \ \sigma_{\sigma_{a}(b)}(\sigma_{\tau_{b}(a)}(c))= \sigma_a(\sigma_b(c)).  
\label{basiko} \nonumber
\end{equation}
If in addition for all $a,b \in X,$ $h_a = h_b \Rightarrow a =b,$ then for all $a,b,c \in X,$ (\ref{basic00}) holds.

Also, assume that $\sigma_a(b) = \sigma_a(c),$ then $h_{\sigma_a(b)}w_a = h_{\sigma_a(c)}w_a$ and by the last relation in (\ref{qualgbb}) we obtain
$w_a h_b= w_ah_c,$ which due to the invertibility of $w_a$ leads to $h_b=h_c$ and hence $b=c.$
\end{proof}
\begin{definition} The basic set-theoretic YB algebra ${\cal A}$ is called a special set-theoretic Yang-Baxter algebra if 
for all $a \in X,$ $\sigma_a:X \to X$ is a bijection. 
\end{definition}
Notice that the element $c= \sum_{a\in X} h_a,$ is a central element of the special set-theoretic YB
algebra ${\cal A}$. This can be immediately shown by means of the definition of the algebra ${\cal A}.$ 
We consider henceforth, without loss of generality, $c=1_{{\cal A}}$. 

\begin{remark} \label{rem0} Throughout this manuscript we will be considering the following notation.
Let $X = \{x_1, x_2, \ldots, x_n\}$ and consider the vector space $V= \mathbb{C}X$ of dimension equal to the cardinality of $X.$ Also ${\mathbb B} = \{e_x\},~x\in X$ is the standard canonical basis of the $n$-dimensional vector space ${\mathbb C}^n,$ that is $e_{x_j}$ is the $n$-dimensional column vector with 1 in the $j^{th}$ row and zeros elsewhere. 
Let also ${\mathbb B}^*= \{e_x^T\},~x\in X$ ( $^T$ denotes transposition) be the dual basis: $e_x^T e_y= \delta_{x,y},$ also $e_{x,y} := e_x  e_y^T$ ($n \times n$ matrices), $x,y \in X$ and they form a basis of $\mbox{End}({\mathbb C}^n).$ That is, to each finite set $X$ we associate a vector space of dimension equal to the cardinality of $X,$ so that each element of the set is represented by a vector of the basis and each map within the set is represented as an $n  \times n$ matrix.
\end{remark}

\begin{remark}[\bf Fundamental representation of the special set-theoretic YB algebra:] \label{remfu2b}  
Let ${\cal A}$ be the special set-theoretic YB algebra and $\rho:  {\cal A} \to \mbox{End}({\mathbb C}^n),$ such that
\begin{equation}
 h_a\mapsto e_{a,a}, \quad  
w_a \mapsto \sum_{b \in X} e_{\sigma_a(b),b}.\label{repbb1}
\end{equation}
Indeed, it can be verified that the above represented elements satisfy the algebraic relations of the special set-theoretic YB algebra (\ref{qualgbb}) if and only if $\sigma_a(\sigma_b(c))= \sigma_{\sigma_{a}(b)}(\sigma_{\tau_{b}(a)}(c))$ for all $a,b,c \in X$.
\end{remark}

\begin{theorem}[Hopf algebra]
 \label{basica2b} 
Let ${\cal A}$ 
be the special set-theoretic YB algebra and \[+: X \times X \to X, \quad (a,b)\mapsto a +b.\]
If $(X, +, 0)$ is a group and for all $a,b,x \in X,$ 
\begin{equation}
 \sigma_x(a) + \sigma_x(b) = \sigma_x(a + b), \label{condition0}
\end{equation} 
then, $({\cal A}, \Delta, \epsilon, s)$ is a Hopf algebra with:
\begin{enumerate}
    \item \textbf{Co-product:} $\Delta: {\cal A} \to {\cal A} \otimes {\cal A}$, \[\Delta(w_a^{\pm 1}) = 
w_a^{\pm 1}\otimes w_a^{\pm 1} \quad \text{and} \quad \Delta(h_a) = \sum_{b, c \in X} h_b \otimes h_c \big |_{b+c =a}.\]
\item \textbf{Co-unit:} $\epsilon: {\cal A} \to k,$  $~\epsilon(w_a^{\pm 1}) =1$ and $\epsilon(h_a) = \delta_{a,0}$
\item \textbf{Antipode:} $s: {\cal A} \to {\cal A},$ $~s(w_a^{\pm 1}) =w_a^{\mp 1}$ and $s(h_a) = h_{-a},$ where $-a$ is the inverse of $a\in X,$ in $(X,+).$ 
\end{enumerate} 
The opposite is also true, i.e. if $({\cal A}, \Delta, \epsilon, s)$ is a Hopf algebra with coproducts given above (part 1 of the Theorem), then $(X, +, 0)$ is a group and (\ref{condition0}) holds.
\end{theorem}
\begin{proof}  We first assume that $(X,+,0)$ is a group and (\ref{condition0}) holds.
To prove that $({\cal A}, \Delta, \epsilon, s)$ is a Hopf algebra, we show that all the axioms of Definition \ref{hopf} are satisfied (see also \cite{DoRySt} for a general proof).
\begin{itemize}
\item The coproduct $\Delta$ is an algebra homomorphism.  Indeed, the coproducts satisfy the algebraic relations (\ref{qualgbb}). Specifically, we show that $\Delta(w_a) \Delta(h_b)= \Delta(h_{\sigma_a(b)}) \Delta(w_a)$ for all $a,b \in X$, by using (\ref{condition0}).

\item The coproducts are coassociative: for all $a\in X,$ $w_a$ is a group-like element, so co-associativity obviously holds. For $h_a$ co-associativity holds, due to the associativity of $+,$ recall $(X, +)$ is a group. 

\item It immediately follows for the group-like elements, 
\[(\id \otimes \epsilon)\Delta(w_a^{\pm 1}) = (\epsilon \otimes \id)\Delta(w_a^{\pm 1}) = w_a^{\pm 1}.\]
Also,
$(\epsilon \otimes \id )\Delta(h_a) =\sum_{b,c \in X} \delta_{b,0} h_c\big |_{b+c =a} = h_a,$ similarly $(\id \otimes \epsilon)\Delta(h_a) = h_a,$ for all $a\in X.$

\item For the group-like elements, $m(s \otimes \id)\Delta(w_a^{\pm 1}) =m(\id \otimes s) \Delta(w_a^{\pm 1})= 1_{\cal A}.$
Moreover, \[m(s \otimes \id)\Delta(h_a) = \sum_{b, c\in X} h_{-b} h_c\big |_{b+c =a} = \delta_{a,0} 1_{\cal A} = \epsilon(h_a) 1_{\cal A},\] where we have used that $h_b h_c = \delta_{b,c} h_b$ and $\sum_{b\in X} h_b =1_{\cal A}.$ 
Similarly, \[m(\id \otimes s)\Delta(h_a) =\epsilon(h_a) 1_{\cal A}.\]
\end{itemize}
To prove the opposite, i.e. if $({\cal A}, \Delta, \epsilon, s)$ is a Hopf algebra then $({X,+,0})$ is a group and (\ref{condition0}) holds is straightforward; we follow the logic of the proof above backwards. We also recall that given the coproduct in a Hopf algebra the counit and antipode can be uniquely derived by the axioms of a Hopf algebra (see for instance \cite{Majid}).
\end{proof} 

\begin{remark} From the proof of Theorem \ref{basica2b} it follows: by requiring $({\cal A}, \Delta, \epsilon)$ to be a bialgebra with coproducts given by (2) in Theorem \ref{basica2b}, we conclude that $(X, +, 0)$ is a monoid and also $\sigma_x(a) + \sigma_x(b) = \sigma_x(a+b),$ for all $a,b,x \in X.$ If we further require $({\cal A}, \Delta ,\epsilon, s)$ to be a Hopf algebra then $(X, +, 0)$ is a group.

Note that whenever $(X,+)$ forms an abelian group, the corresponding Hopf algebra $({\cal A}, \Delta, \epsilon, s)$ is co-commutative, which means that the opposite coproduct coincides with the coproduct itself, i.e. $\Delta^{(op)}=\Delta$.\end{remark}

\subsection{Set-theoretic Drinfel'd twist}

In this subsection we introduce the set-theoretic (or combinatorial) Drinfel'd twist (see \cite{Doikoutw, DoGhVl, DoRySt} and a relevant construction in \cite{Sol}). Using the twist, we will be able to obtain the universal set-theoretic ${\cal R}$-matrix associated with the special set-theoretic YB algebra.

Before we introduce the set-theoretic twist, we recall a general statement \cite{Drinfeld}.
\begin{proposition}[Drinfel'd] \label{prot}  Let ${\cal A}$ be a unital, associative algebra, ${\cal F}, {\cal R} \in {\cal A} \otimes {\cal A}$ be invertible elements, ${\cal R}$ satisfy the Yang-Baxter equation (\ref{UYBE}), and ${\cal F}_{1,23}, {\cal F}_{12,3} \in {\cal A}^{\otimes 3}$ be such that
\begin{enumerate}
\item  ${\cal F}_{23}{\cal F}_{1,23} = {\cal F}_{12} {\cal F}_{12,3},$ where recall ${\cal F}_{12} = {\cal F} \otimes 1_{\cal A}$ and ${\cal F}_{23} = 1_{\cal A} \otimes {\cal F}.$ 
\item  ${\cal F}_{1,32}  {\cal R}_{23} = {\cal R}_{23} {\cal F}_{1,23}$ and ${\cal F}_{21,3} {\cal R}_{12} = {\cal R}_{12} {\cal F}_{12,3}.$
\end{enumerate}
 That is, ${\cal F}$ is an admissible Drinfel'd twist.
 Define also ${\cal R}^F := {\cal F}^{(op)} {\cal R} {\cal F}^{-1},$ ${\cal F}^{(op)} = \pi({\cal F})$ where $\pi: {\cal A} \otimes {\cal A} \to {\cal A} \otimes {\cal A}$ is the flip map. Then ${\cal R}^F$ also satisfies the Yang-Baxter equation.
\end{proposition}
\begin{proof}

It is convenient to introduce some handy notation that can be used in the following.
First, let ${\cal F}_{123} := {\cal F}_{12}{\cal F}_{1,23} = {\cal F}_{23} {\cal F}_{1,23}.$
Let also $i,j,k \in \{1,2,3\},$ then ${\cal F}_{jik} = \pi_{ij}( {\cal F}_{ijk})$ and
${\cal F}_{ikj} = \pi_{jk}({\cal F}_{ijk}),$ where $\pi$ is the flip map. This notation 
describes all possible permutations of the indices $1,\ 2,\ 3$. 

The proof is quite straightforward, \cite{Drinfeld},
we just give a brief outline here. 
We first prove that ${\cal F}_{jik}{\cal R}_{ij}{\cal F}^{-1}_{ijk} = {\cal R}^F_{ij},$ indeed via condition (2) of the proposition the definition of ${\cal R}^{F}$ and the notation introduced above we have
\begin{equation}
{\cal F}_{jik}{\cal R}_{ij}{\cal F}^{-1}_{ijk} = {\cal F}_{ji} {\cal F}_{ji,k} {\cal R}_{ij} {\cal F}^{-1}_{ijk} = {\cal F}_{ji}{\cal R}_{ij}{\cal F}_{ij,k}{\cal F}_{ijk}^{-1} = {\cal R}^{F}_{ij} {\cal F}_{ij}{\cal F}_{ij,k}{\cal F}^{-1}_{ijk} = {\cal R}^F_{ij}.
\end{equation}
Similarly, it is shown that ${\cal F}_{ikj}{\cal R}_{jk}{\cal F}^{-1}_{ijk} = {\cal R}^F_{jk}.$

Then from the YBE we have (see also \cite{Doikoutw}),
\begin{equation}
{\cal  F}_{321} {\cal R}_{12} {\cal R}_{13} {\cal R}_{23} ={\cal  F}_{321} 
{\cal R}_{23}{\cal R}_{13}{\cal R}_{12}\ \Rightarrow\  {\cal R}^F_{12} {\cal R}^F_{13} {\cal R}^F_{23}{\cal  F}_{123} 
=  {\cal R}^F_{23}{\cal R}^F_{13}{\cal R}^F_{12}{\cal  F}_{123}. \nonumber
\end{equation}
But ${\cal F}_{123}$ is invertible, hence ${\cal R}^F$ indeed satisfies the Yang-Baxter equation.
\end{proof}

\begin{theorem}[Set-theoretic twist \texorpdfstring{\cite{Doikoutw, DoGhVl, DoRySt}}{}] \label{twist2}  
Let ${\cal A}$
be the special set-theoretic YB algebra and ${\cal F}\in {\cal A} \otimes {\cal A},$  
such that ${\cal F} =\sum_{b\in X}h_b \otimes w_b^{-1},$
${\cal R}^F_{ij} := {\cal F}_{ji} {\cal F}_{ij}^{-1},$ $i,j \in \{1,2,3\}.$
We also define:
\begin{eqnarray}
{\cal F}_{1,23} := \sum_{a\in X} h_a\otimes w_a^{-1} \otimes 
w_a^{-1},  \quad {\cal F}_{12,3} : =  \sum_{a,b \in X}h_a\otimes h_{\sigma_a(b)} 
\otimes w^{-1}_{b} w^{-1}_a. \label{deff}
\end{eqnarray}
Let $\sigma_a, \tau_b: X \to X,$
be such that $\sigma_{\sigma_a(b)}(\tau_b(a))= a$ for every $a,b \in X.$  
Then, the following statements are true:
\begin{enumerate}
\item ${\cal F}_{12} {\cal F}_{12,3} ={\cal F}_{23} {\cal F}_{1,23} =: {\cal F}_{123}.$
\item For  $i,j,k \in \{1,2,3\}$:
(i) ${\cal F}_{ikj}{\cal F}_{ijk}^{-1} ={\cal R}^F_{jk} $
and 
(ii)  ${\cal F}_{jik} {\cal F}^{-1}_{ijk} ={\cal R}^F_{ij}.$
\end{enumerate}
That is, ${\cal F}$ is an admissible Drinfel'd twist.
\end{theorem}
\begin{proof}
The proof is straightforward based on the underlying algebra ${\cal A}$ 
(the detailed proof can also be found in \cite{Doikoutw, DoRySt}).
\begin{enumerate}
\item Indeed, this is proven by a direct computation and use of the special set-theoretic YB algebra. 
In fact, ${\cal F}_{123} = \sum_{a,b\in X} h_a \otimes h_b w_a^{-1} \otimes w_b^{-1} w_a^{-1}.$
\item Given the notation introduced before in the proof of Proposition \ref{prot} it  suffices to show  that
${\cal F}_{132} {\cal F}^{-1}_{123}= {\cal R}^F_{23} $ and ${\cal F}_{213} 
{\cal F}_{123}^{-1} = {\cal R}^F_{12} .$ 

We first show that ${\cal F}_{1,23} = {\cal F}_{1,32}$, which is straightforward from the definition in (\ref{deff}); notice that ${\cal F}_{1,23} = (\id \otimes \Delta){\cal F}$. Also,
\begin{eqnarray}
{\cal F}_{12,3} &=& \sum_{a,b \in X}h_a\otimes h_{\sigma_a(b)} 
\otimes (w_{a} w_b)^{-1} = \sum_{a,b \in X}h_a\otimes h_{\sigma_a(b)} 
\otimes (w_{\sigma_a(b)} w_{\tau_b(a)})^{-1}\nonumber \\ &=& \sum_{\hat a, \hat b\in X} h_{\sigma_{\hat a}(\hat b)} \otimes h_{\hat a} 
\otimes (w_{\hat a} w_{\hat b})^{-1} = {\cal F}_{21,3},
\end{eqnarray}
where we have set in the equation above $\hat a := \sigma_{a}(b)$ and $\hat b : = \tau_b(a)$ which leads to 
$a = \sigma_{\hat a}(\hat b)$ due to $\sigma_{\sigma_a(b)}(\tau_b(a)) = a.$

It then immediately follows (see also Proposition \ref{prot}):
\[{\cal F}_{132} {\cal F}_{123}^{-1} = {\cal F}_{32} {\cal F}_{1,32} {\cal F}_{1,23}^{-1}{\cal F}_{23}^{-1}= 
{\cal F}_{32} {\cal F}_{23}^{-1} = {\cal R}_{23}^F. \]
\[ {\cal F}_{213} {\cal F}_{123}^{-1} = {\cal F}_{21} {\cal F}_{21,3} {\cal F}_{12,3}^{-1} {\cal F}_{12}^{-1}
={\cal F}_{21} {\cal F}_{12}^{-1}  = {\cal R}^F_{12}. \]
\end{enumerate}
Due to Proposition \ref{prot}, we also deduce that ${\cal R}^F$ is a solution of the Yang-Baxter equation.
\end{proof}

\begin{remark}[Twisted universal \texorpdfstring{${\cal R}$-matrix}{R-matrix}] \label{remtw}
 We derive explicit
expressions of the twisted universal ${\cal R}$-matrix and the twisted coproducts
of the algebra. We recall the admissible twist ${\cal F} = \sum_{b\in X}h_b \otimes w_b^{-1}.$
\begin{itemize}

\item The twisted ${\cal R}$-matrix: 
\[{\cal R}^F = {\cal F}^{(op)} {\cal F}^{-1} = \sum_{a,b\in X}h_bw_a^{-1}  
 \otimes h_a w_{\sigma_a(b)}.\]

\item The twisted coproducts: $\Delta_F(y) = {\cal F} \Delta(y) {\cal F}^{-1},$ $y\in {\cal A}$ 
and recall from Theorem~\ref{basica2b}, that $(X,+)$ is a group and  for all $a\in X,$
\[ \Delta(w_a) =w_a\otimes w_a,\quad \Delta(h_a) = 
\sum_{b,c\in X} h_b \otimes h_c\big |_{b+c =a}.\]
Then, the twisted coproducts read as:
\begin{equation} \Delta_F(w_a) = \sum_{b \in X} w_a h_b\otimes w _{\tau_b(a)}, ~~~
\Delta_F(h_a) = \sum_{b,c \in X} h_b \otimes  h_c \big |_{b + \sigma_b(c) =a}, \label{dtw}
\end{equation}
\end{itemize}
and recall  $\tau_{b}(a):=\sigma_{\sigma_a(b)}^{-1}(a),$ hence ${\cal R}^F_{12} {\cal R}_{21}^F = 1_{{\cal A} \otimes {\cal A}}.$  It also follows that $${\cal R}^F \Delta_F(Y) = \Delta_F^{(op)}(Y) {\cal R}^F, Y \in {\cal A}$$ if $(X, +)$ is an abelian group (see a detailed proof in Theorem \ref{twist2b}).
\end{remark}

\begin{remark}[\bf Fundamental representation $\&$ the set-theoretic solution:]  \label{remfu2bc}  
Let ${\cal A}$ be the special set-theoretic algebra and $\rho:  {\cal A} \to \mbox{End}({\mathbb C}^n),$ such that
\begin{equation}
h_a\mapsto e_{a,a}, \quad  
w_a \mapsto \sum_{b \in X} e_{\sigma_a(b),b}.\label{repbb2}
\end{equation}
Moreover, ${\cal F} \mapsto F:= \sum_{a,b \in X} e_{a,a} \otimes e_{b, \sigma_a(b)}$ and $ {\cal R}^F \mapsto R^F:= 
\sum_{a,b\in X} e_{b,\sigma_a(b)} \otimes e_{a, \tau_b(a)},$
where we recall that for all $a,b,c \in X,$ $\sigma_{\sigma_{a}(b)}(\sigma_{\tau_{b}(a)}(c)) = \sigma_a(\sigma_b(c))$ (see also Proposition~\ref{qua2}), $\tau_{b}(a):=\sigma_{\sigma_a(b)}^{-1}(a)$ and $R^F_{12} R^F_{21} = 1_{n^2},$ where $ 1_{n^2}$ is 
the $n^2$ dimensional identity matrix. Then, $F$ is a combinatorial twist and $R^F$ is a combinatorial (set-theoretic) solution of the Yang-Baxter equation.
\end{remark}
 
We present below the $n$-fold twist (see also \cite{Doikoutw, Doikoupar}).
\begin{lemma}[The \texorpdfstring{$n$-fold}{n-fold} twist] \label{nfold0}   
Let ${\cal A}$ be the special set-theoretic YB algebra and let ${\cal F} \in {\cal A} \otimes {\cal A}$ be such that 
${\cal F} = \sum_{a \in X} h_a\otimes w_a^{-1}.$
Define also,
\begin{eqnarray}
  {\cal F}_{1, 2 3\ldots n} &:=& \sum_{a\in X} h_a \otimes \Delta^{(n-1)}(w_a^{-1}) =\sum_{a\in X} h_a \otimes w_a^{-1}  \otimes w_a^{-1} \otimes   \ldots \otimes w_a^{-1}, \nonumber\\ 
    {\cal F}_{1 2 \ldots n-1, n} &:=& \sum_{a_1, a_2, \ldots, a_{n-1} \in X} 
   h_{a_1} \otimes h_{\sigma_{a_1}(a_2)} \otimes h_{\sigma_{a_1}(\sigma_{a_2}(a_3))} \otimes \ldots \nonumber \\ 
   & \otimes &
   h_{\sigma_{a_1}(\sigma_{a_2}(\ldots 
   \sigma_{a_{n-2}}(a_{n-1})) \ldots )} 
\otimes w_{a_{n-1}}^{-1}
w_{a_{n-2}}^{-1} \ldots w_{a_{1}}^{-1}. \nonumber 
   \end{eqnarray}
Then,
\begin{enumerate}
\item ${\cal F}_{2\ldots n}
{\cal F}_{1,2\ldots n} =
{\cal F}_{12\ldots n-1} {\cal F}_{12\ldots n-1,n}=: {\cal F}_{12\ldots n}.$ 
\item The explicit expression of the $n$-fold twist is given as
\begin{eqnarray}
{\cal F}_{12\ldots n} 
&=&
\sum_{a_1, a_2, \ldots, a_{n-1} \in X} h_{a_1} \otimes h_{a_2} w_{a_1}^{-1} \otimes 
h_{a_3} w_{a_2}^{-1}w_{a_1}^{-1} \otimes \ldots \otimes \nonumber \\
& & h_{a_{n-1}}w_{a_{n-2}}^{-1} \ldots w_{a_1}^{-1} \otimes w_{a_{n-1}}^{-1}w_{a_{n-2}}^{-1} \ldots w_{a_1}^{-1}. 
 \label{nfold}
\end{eqnarray}

\item ${\cal F}_{1,2 3 \ldots j+1 j\ldots n}  =  {\cal F}_{1,2 3 \ldots j j+1\ldots n},$ $~n-1 \geq j>1,$
${\cal F}_{12 \ldots j+1 j\ldots n-1, n} = {\cal F}_{1 2 \ldots j j+1\ldots n-1,  n},$ $~n-1>j\geq 1,$
${\cal F}_{12 \ldots j+1 j\ldots n}  = {\cal R}^{F}_{jj+1} {\cal F}_{1 2 \ldots j j+1\ldots n},$ $~n-1 \geq j\geq 1.$ \end{enumerate}
\end{lemma}
\begin{proof} 
These statements are proven by iteration and direct computation using the ${\cal A}$ algebra relations. Part (2) of Theorem \ref{twist2} is also used in proving (3) (see also \cite{Doikoutw, Doikoupar}).
\end{proof}

\subsection{The twisted Hopf algebra}
Motivated by Theorem \ref{basica2b} on the conditions that make the special set-theoretic YB algebra  
a Hopf algebra and by the twisted coproducts (\ref{dtw}) in Remark \ref{remtw} we prove the following Theorem (see also \cite{DoiRyb} for relevant results). Notice in particular the condition for all $a,b,c \in X,$ $b+\sigma_b(c) = a$ that appears in $\Delta_F(h_a) = \sum_{b,c \in X} h_b \otimes h_c\big|_{b + \sigma_{b}(c) =a}.$ It is thus natural to introduce a new binary operation, $\circ: X \times X \to X,$ such that $a\circ b : = a + \sigma_a(b),$ for all $a,b \in X.$

\begin{theorem} \label{corf0}
    Let ${\cal A}$ be the special set-theoretic YB algebra, $(X,+, 0)$  a group and for all $a,b \in X,$ $\sigma_a,\tau_b: X \to X,$ such that $\sigma_{\sigma_a(b)}(\tau_b(a)) = a.$  Let also for all $a,b \in X,$ $a\circ b := a + \sigma_a(b).$    
    \begin{enumerate}
    \item Then for all $a,b \in X,$ $\sigma_a(b) \circ \tau_b(a) = -a +a\circ b +a.$ 
    \item If in addition $(X, \circ)$ is a semigroup and for all $a,b,c \in X,$ $\sigma_a(b+c) = \sigma_a(b) + \sigma_a(c),$ then for all $a,b,c \in X,$ 
    \begin{enumerate}
    \item $\sigma_a(\sigma_b(c)) = \sigma_{a\circ b}(c).$ 
    \item $(X, \circ, 0)$ is a group.
    \item $a\circ (b+c) = a\circ b- a + a\circ c.$   
    \item $\sigma_a(0) = 0,$ $\tau_0(a) = a$ and $\sigma_0(a) = a,$ $\tau_a(0) = 0.$  
   \end{enumerate}
   \end{enumerate}
\end{theorem}
\begin{proof}

$ $

\begin{enumerate}

\item Recall $\sigma_{\sigma_a(b)}(\tau_b(a)) =a,$ and 
$$  a\circ b\  = a+ \sigma_{a}(b) ~~\mbox{and} ~~ \sigma_a(b)\circ \tau_b(a)=  \sigma_{a}(b) + \sigma_{\sigma_a(b)}(\tau_b(a)) =\sigma_{a}(b) +a.$$
The two equations above lead to $-a+a\circ b +a=\sigma_a(b)\circ \tau_b(a).$

\item We now assume that $(X, \circ)$ is a semigroup and $\sigma_a$ is a $(X,+)$ group homomorphism for all $a \in X.$
\begin{enumerate}
 \item From associativity in $(X, \circ)$:
\begin{eqnarray}
  &&  (a\circ b)\circ c = a\circ b + \sigma_{a\circ b}(c)~~~~~\mbox{and}\nonumber \\
    && a\circ(b\circ c) = a +\sigma_a(b\circ c) = a+ \sigma_a(b+\sigma_b(c)) =\nonumber\\
    && a + \sigma_a(b) + \sigma_{a}(\sigma_b(c))= a\circ b +\sigma_{a}(\sigma_b(c)). \nonumber
\end{eqnarray}
From the two equations above we conclude that $\sigma_a(\sigma_b(c)) =\sigma_{a\circ b}(c).$

\item From $\sigma_a(\sigma_b(c)) =\sigma_{a\circ b}(c)$ we obtain for all $a,b, c \in X$
\begin{eqnarray}
&& -a + a\circ \sigma_b(c) = -a\circ b  + a\circ b \circ c \Rightarrow \nonumber \\
&& a\circ(-b +b\circ c) = a - a \circ b + a\circ b \circ c. \label{e1} \nonumber
\end{eqnarray}
That is for all $a,b,c \in X,$ 
\begin{equation}
    a\circ (-b +c)= a -a\circ b + a\circ c. \label{e00}
\end{equation}

 There is a right neutral element.
From the distributivity condition above,
\begin{eqnarray}
    a\circ (-0+b) = a \circ b \Rightarrow a -a\circ 0 + a\circ b = a \circ b \Rightarrow a\circ 0 =a. \nonumber
\end{eqnarray}

 Also, from the bijectivity of $\sigma_a$ for all $a \in X:$
\begin{eqnarray}
    \sigma_a(b) = \sigma_b(c) \Rightarrow b =c \nonumber
\end{eqnarray}
which leads to 
$$a\circ b = a \circ c \Rightarrow b =c,$$
i.e. left cancellation holds.

There is a unique right inverse in $(X, \circ),$ indeed, $a^{-1} := \sigma_a^{-1}(-a)$ for all $a \in X,$ then 
$$a \circ a ^{-1} =a + \sigma_a(\sigma^{-1}_a(-a)) = a -a = 0,$$ which is also a left inverse.
Indeed, from the definition of the inverse, $a^{-1}\in X,$ so there is a unique right inverse element for $a^{-1}$ denoted as $(a^{-1})^{-1},$
then 
$$a^{-1} \circ (a^{-1})^{-1} = 0\ \Rightarrow  (a^{-1})^{-1} = a \circ 0 =a\ \Rightarrow a^{-1}\circ a =0.$$
Also,
$$a \circ 0 \circ a^{-1} = 0 \Rightarrow a \circ 0 = 0 \circ a,$$ i.e. $0$ is also a left neutral element in $(X, \circ).$ 
And we conclude that $(X, \circ)$ is a group.


\item From the distributivity condition (\ref{e00}),
\begin{eqnarray}
    a\circ (-b +0) = a\circ (-b) \Rightarrow a -a \circ b + a = a\circ (-b). \label{e2} 
\end{eqnarray}
Then, from (\ref{e00}), (\ref{e2}):
\begin{eqnarray}
a\circ (b+c) = a - a\circ (-b) + a\circ c = a\circ b -a + a\circ c. \label{dis} \nonumber
\end{eqnarray}

 \item These equalities follow from expressions \[\sigma_a(b) = -a + a\circ b, -a +a\circ b +a = \sigma_a(b) \circ\tau_b(a)\] and the fact that both $(X, +, 0),$ $(X, \circ, 0)$ are groups.
 \hfill \qedhere
 \end{enumerate}
 \end{enumerate}
\end{proof}
Notice that if $(X, +)$ is an abelian group, then $a\circ b = \sigma_a(b) \circ \tau_b(a).$

 Algebraic structures as the one derived in Theorem \ref{corf0}, where $X$ is a non-empty set equipped with two group operation $+, \circ,$ such that $a\circ(b+c) = a \circ b -a +b\circ c,$ for all $a,b,c \in X$ are known as  skew {\it left braces} \cite{Rump1, Rump2, GV}. If $(X, +)$ is abelian then the structure is called a left brace. Braces were introduced by Rump \cite{Rump1, Rump2, JeOk} in the context of finding involutive set-theoretic solutions of the Yang-Baxter equation. The precise definition of (skew) braces is given below.
 \begin{definition}  \label{defbrace} 
A  skew left brace is a set $X$ together with two group operations \[+,\circ :X\times X\to X.\]
The $+$ operation is called addition and $\circ$ is called multiplication, such that for all $ a,b,c\in B$,
\begin{equation}\label{def:dis}
a\circ (b+c)=a\circ b-a+a\circ c.
\end{equation}
\end{definition}
If $(X,+)$ is an abelian group, then $(X, +, \circ)$ is called a left brace. In this paper, whenever we say (skew) brace, we mean a (skew) left brace. Recall also that for every (skew) brace $0=1,$ where $0$ is the neutral element in $(X, +)$ and $1$ is the neutral element in $(X,\circ)$.

\begin{lemma} \label{centr1}
    Let ${\cal A}$ 
be the special set-theoretic YB algebra. Let also $(X, +, \circ)$ be a skew brace and for all $a,b \in X,$ $\sigma_a, \tau_b: X\to X,$ such that
\begin{equation}
 \sigma_a(b) = - a +a \circ b, ~~~~\sigma_{\sigma_a(b)}(\tau_b(a)) =a. \label{condition0b}
\end{equation} 
Then $w_0$ is a central element in ${\cal A},$ where $0$ is the neutral element in $(X, +)$ and $(X, \circ).$
\end{lemma}
\begin{proof}
    Note that $\sigma_0(b) =b,$ $\tau_b(0) =0,$ then from (\ref{qualgbb}) for all $b\in X,$ 
    \[w_0 w_b = w_{\sigma_0(b)} w_{\tau_b(0)} = w_{b} w_0 ~~\mbox{and} ~~w_0h_b = h_{\sigma_0(b)} w_0 = h_bw_0.\qedhere\]
    \end{proof}

\begin{theorem} \label{twist2b} Let ${\cal A}$ 
be the special set-theoretic YB algebra. If $(X, +, \circ)$ is a brace and for all $a,b \in X,$ $\sigma_a,\ \tau_b : X \to X,$ such that $\sigma_a(b)= -a +a \circ b$ and $\sigma_{\sigma_a(b)}(\tau_b(a)) =a,$ then $({\cal A}, \Delta_F, \epsilon, \tilde s)$ is a  Hopf algebra, where the twisted coproducts are given in 
Remark \ref{remtw}, $\epsilon$ is given in Theorem \ref{basica2b} and $\tilde s:  {\cal A} \to {\cal A},$ such that for all $a \in 
 X,$ 
 \begin{equation}
 \tilde s(h_a) = h_{a^{-1}}, ~~~
 \tilde s(w_a) = \sum_{b\in X} h_bw^{-1}_{\tau_{b^{-1}}(a)}, \label{antip}
 \end{equation}
 $a^{-1}$ is the inverse of $a \in X,$ in the group $(X, \circ).$ 
If in addition for all $a,b \in X,$ $w_a w_b = w_{a\circ b}$ and ${\cal R}^F$ is given in Remark \ref{remtw}, then 
 $({\cal A}, \Delta_F, \epsilon, \tilde s, {\cal R}^F)$ is a quasi-triangular Hopf algebra.
\end{theorem}
\begin{proof}
This is a consequence of Theorem \ref{basica2b}, Remark \ref{remtw} and Lemma \ref{centr1}. 

We first prove the coassociativity of the twisted coproducts; indeed, due to the associativity in $(X, \circ),$ for all $a \in X,$ 
$$(\Delta_F \otimes \id) \Delta_F(h_a)=(\id \otimes\Delta_F)\Delta_F(h_a) = \sum_{b,c,d \in X} h_{b}\otimes h_c \otimes h_d|_{b\circ c \circ d = a}.$$  
Also, due to $\tau_b(a) = \sigma_{\sigma_a(b)}^{-1}(a)$ and $\sigma_a(b) = -a + a\circ b,$
$$ (\Delta_F \otimes \id)\Delta_F(w_a) =(\id \otimes\Delta_F)\Delta_F(w_a)= \sum_{b,c \in X} w_a h_b \otimes w_{\tau_b(a)}h_c \otimes w_{\tau_{b\circ c}(a)}.$$
Moreover, we observe that $(\epsilon \otimes \id){\cal F} = w_0,$ recall from Lemma \ref{centr1} that $w_0$ is central in ${\cal A};$ also $(\id \otimes \epsilon){\cal F} = 1_{\cal A},$ which lead to:
$$(\epsilon \otimes \id)\Delta_F(x) = (\id \otimes \epsilon)\Delta_F(x)= x, ~~~x\in {\cal A}.$$ This concludes our proof that $({\cal A}, \Delta_F, \epsilon)$ is a bialgebra. Moreover, from the form of the antipode $\tilde s$ (\ref{antip}), we show that
\begin{equation}
m(\tilde s \otimes \id)\Delta_F(x) = m (\id \otimes \tilde s)\Delta_F(x) = \epsilon(x) 1_{\cal A}, ~~~~x \in {\cal A}.\nonumber
\end{equation}
And this concludes that proof that $({\cal A}, \Delta_F, \epsilon, \tilde s)$ is a Hopf algebra.

To show that $({\cal A}, \Delta_F, \epsilon, \tilde s, {\cal R}^F)$ is a quasi-triangular Hopf algebra we also need to show conditions (1) and (2) of Definition 1.2. The Hopf algebra $({\cal A}, \Delta, \epsilon, s),$ is cocommutative due the fact that $(X, +)$ is an abelian group, i.e. for $x\in {\cal A,}$ $\Delta^{(op)}(x) = \Delta(x)$ and for ${\cal F}$ being the admissible twist of Theorem \ref{twist2}:
$${\cal F}^{(op)} \Delta^{(op)}(x)({\cal F}^{(op)})^{-1} {\cal F}^{(op)} {\cal F}^{-1} ={\cal F}^{(op)}{\cal F}^{-1} {\cal F}\Delta(x) {\cal F}^{-1} \Rightarrow \Delta_F^{(op)}(x){\cal R}^F = {\cal R}^F \Delta_F(x).$$

From the algebraic relations of the special set-theoretic YB algebra and recalling that \[{\cal R}^F = \sum_{a,b \in X} h_b w_a^{-1} \otimes h_a w_{\sigma_a(b)}, a\circ b = \sigma_a(b) \circ \tau_b(a), \sigma_a(b) = -a + a\circ b,\] and $w_a w_b = w_{a\circ b},$ we deduce
\begin{eqnarray}
{\cal R}_{13}^F{\cal R}_{23}^F &=&  \sum_{a,b,c \in X} h_bw_a^{-1} \otimes h_c w^{-1}_{\tau_b(a)} \otimes h_a w_{\sigma_a(b)} w_{\sigma_{\tau_b(a)}}(c) \nonumber \\ &=& \sum_{a,b,c \in X} h_bw_a^{-1} \otimes h_c w^{-1}_{\tau_b(a)} \otimes h_a w_{\sigma_a(b\circ c)} = (\Delta_F \otimes \id){\cal R}^F. \nonumber
\end{eqnarray}
Similarly,
\begin{eqnarray}
{\cal R}_{13}^F{\cal R}_{12}^F &=&  
\sum_{a,b,\hat a, \hat b \in X} h_bw_a^{-1}h_{\sigma_a(\hat b)} w_{\hat a}^{-1} \otimes h_{\hat a} w_{\sigma_{\hat a}(\sigma_a(\hat b))} \otimes h_a w_{\sigma_a(b)}  \nonumber \\ &=& 
\sum_{a,\hat a, b \in X} h_b w_{\hat a \circ a}^{-1} \otimes h_{\hat a} w_{\sigma_{\hat a \circ a}(b)} \otimes h_a w_{\sigma_a(b)}. \label{RR}
\end{eqnarray}
Also,
\begin{eqnarray}
(\id \otimes \Delta_F){\cal R}^F &=& \sum_{a,b\in X}h_b w_a^{-1} \otimes \Delta(h_a w_{\sigma_a(b)})  \nonumber\\  &=&\sum_{a_1\circ a_2 =a,b,c \in X} h_b w_a^{-1} \otimes h_{a_1}w_{\sigma_a(b)} \otimes h_{a_2} w_{\tau_{c}(\sigma_a(b))}\big|_{a_1 = \sigma_{\sigma_a(b)}(c)}. \label{delta2}
\end{eqnarray}
From the condition $a_1 = \sigma_{\sigma_a(b)}(c)$ we deduce $ c =\sigma^{-1}_{\sigma_a(b)}(a_1) = \sigma_{(\sigma_a(b))^{-1}}(a_1),$ we also recall $a\circ b = \sigma_a(b) \circ \tau_b(a),$ and $\sigma_a(b) = -a + a\circ b$  and $a_1\circ a_2 =a$ (\ref{delta2}), which lead to
\begin{eqnarray}
\tau_c(\sigma_{a}(b)) &=& (\sigma_{\sigma_a(b)}(c))^{-1} \circ \sigma_a(b) \circ c = a_1^{-1} \circ \sigma_a(b) \circ \sigma_{(\sigma_a(b))^{-1}}(a_1) \nonumber\\ &=& a_1^{-1}\circ(\sigma_a(b) + a_1) = a_1^{-1}\circ \sigma_a(b) -a_1^{-1} =-a_2 + a_2\circ b = \sigma_{a_2}(b). \label{tau2}
\end{eqnarray}
From equations (\ref{delta2}) and (\ref{tau2}) we conclude,
\begin{equation}
(\id \otimes \Delta_F){\cal R}^F=\sum_{a_1, a_2, b \in X} h_b w_{a_1 \circ a_2}^{-1} \otimes h_{a_1}w_{\sigma_{a_1\circ a_2}(b)} \otimes h_{a_2} w_{\sigma_{a_2}(b)}. \label{dfinal}
\end{equation}
Comparing equation (\ref{dfinal}) with (\ref{RR}) we arrive at ${\cal R}_{13}^F{\cal R}_{12}^F =(\id \otimes \Delta_F){\cal R}^F$. 
And this concludes the second part of our proof that $({\cal A}, \Delta_F, \epsilon, \tilde s, {\cal R}^F)$ is a quasi-triangular Hopf algebra.
    \end{proof}

\begin{remark} Following the proof of Theorem \ref{twist2b} we also conclude:
\begin{enumerate}
\item Assuming $({\cal A}, \Delta, \epsilon,s )$ is a Hopf algebra and requiring  $({\cal A}, \Delta_F, \epsilon)$ to be a bialgebra with coproducts given in Remark \ref{remtw}, we deduce that $(X, \circ)$ is a semigroup. And via Theorem \ref{corf0} we obtain that $(X, +, \circ)$ is a skew brace. Hence, we can define an antipode (\ref{antip}) and $({\cal A}, \Delta_F, \epsilon, \tilde s)$ is a Hopf algebra.

\item Requiring also $({\cal A}, \Delta_F, \epsilon, {\cal R}^F)$ to be a quasi-triangular bialgebra we deduce that $(X, +, \circ)$ is a brace.
\end{enumerate}
\end{remark}
\begin{lemma} \label{corf}
    Let ${\cal A}$ be the special set-theoretic YB algebra and ${\cal F}_{12\ldots n}\in {\cal A}^{\otimes n}$ be the $n$-fold twist (\ref{nfold}) and ${\cal F}_{12\ldots n-1, n}\in {\cal A}^{\otimes n}$ is given in Lemma \ref{nfold0}. Let also $(X, +, \circ)$ be a brace, $\sigma_a,\tau_b: X \to X,$ such that $\sigma_a(b) = -a + a\circ b,$ $\sigma_{\sigma_a(b)}(\tau_b(a)) = a$ and $w_a w_b = w_{a\circ b},$ for all $a,b \in X.$ Then the following statements are true: 
    \begin{enumerate}
   \item ${\cal F}_{12 \ldots n-1,n} = (\Delta^{(n-1)} \otimes \id){\cal F}.$
    \item The $n$-fold twist is given as
    \[{\cal F}_{12\ldots n} = 
    \sum_{a_1, \ldots, a_{n}\in X } h_{a_1} \otimes h_{a_2} w_{a_1}^{-1}\otimes \ldots \otimes h_{a_{n-1}} 
    w_{a_1\circ a_2 \circ \ldots \circ a_{n-2} }^{-1} \otimes w_{a_1\circ a_2 \circ \ldots \circ a_{n-1} }^{-1}.\]

    \end{enumerate}
\end{lemma}
\begin{proof} 
\begin{enumerate}
\item Recall the definition of ${\cal F}_{1,23}$ (\ref{deff}), then 
$${\cal F}_{12,3} = \sum_{a,b\in X} h_a \otimes h_{\sigma_{a}(b)} \otimes 
(w_{a +\sigma_a(b)})^{-1} = \sum_{c\in X} \Delta(h_c) \otimes w_c^{-1} = (\Delta \otimes \id){\cal F}$$
and due to co-associativity ${\cal F}_{12 \ldots n-1,n} =(\Delta^{(n-1)} \otimes \id){\cal F}.$

 \item  This is a consequence of the form of the 
    $n$-twist (\ref{nfold}), relation $w_aw_b = w_{a \circ b}$ for all $a,b \in X$ and the associativity in $(X, \circ).$
    \hfill \qedhere
    \end{enumerate}
    \end{proof}

\section{Twisting the \texorpdfstring{$\mathfrak{gl}_n$}{gl\_n} Yangian} \label{3}

\subsection{Preliminaries: a review on the \texorpdfstring{$\mathfrak{gl}_n$}{gl\_n} Yangian}

We first recall the derivation of quantum groups (or quantum algebras) associated with any given solution $R: V \otimes V \to V \otimes V $ of the (parametric) Yang-Baxter equation (YBE) \cite{Baxter, Yang} (in this manuscript {$V = {\mathbb C}^n$})
\begin{equation}
R_{12}(\lambda_1, \lambda_2)R_{13}(\lambda_1, \lambda_3)R_{23}(\lambda_2, \lambda_3)= R_{23}(\lambda_2, \lambda_3) R_{13}(\lambda_1, \lambda_3) R_{12}(\lambda_1,\lambda_2), \label{frt}
\end{equation}
where $\lambda_1, \lambda_2 \in {\mathbb C}$.  Let $R = \sum_ x {\mathrm a}_x \otimes{\mathrm b}_x,$ ${\mathrm a}_x, {\mathrm b}_x \in \mbox{End}({\mathbb C}^n),$ then in the "index notation": \[ R_{12} =   \sum_x {\mathrm a}_x\otimes {\mathrm b}_x \otimes 1_V, \,  R_{23} = 1_{V} \otimes  \sum_x{\mathrm a}_x \otimes{\mathrm b}_x,\,  \text{ and } R_{13} = \sum_x {\mathrm a}_x \otimes 1_V \otimes{\mathrm b}_x.\]

For the derivation of a quantum algebra associated with a given $R$-matrix we employ the FRT (Faddeev-Reshetikhin-Takhtajan) construction \cite{FRT}.
We recall the standard $n \times n$ matrices $e_{x,y},$ with entries $(e_{x,y})_{z,w} = \delta_{x,z} \delta_{y,w},$ $x,y,z,w \in X,$ and recall $X = \{x_1, x_2, \ldots, x_n\}$ (see also Remark \ref{rem0}).
\begin{definition} \label{frt1}
Let $R(\lambda_1, \lambda_2)\in \mbox{End}(V \otimes V)$ be a solution of the Yang-Baxter equation (\ref{frt}), $\lambda_1, \lambda_2\in {\mathbb C},$ $(V ={\mathbb C}^n).$ Let also \[L(\lambda) : = \sum_{x,y\in X} e_{x,y} \otimes L_{x,y}(\lambda) \in  \mbox{End}(V) \otimes {\mathfrak A},\] where $\lambda \in {\mathbb C}$ and $L_{x,y}(\lambda) =\sum_{m=0}^{\infty}\lambda^{-m} L^{(m)}_{x,y}\in {\mathfrak A}.$ The quantum algebra ${\mathfrak A},$ associated to $R,$ is defined as the quotient of the free unital, associative ${\mathbb C}$-algebra, generated by $\Big \{ L^{(m)}_{x,y}| x,y \in X, \ m \in\{0,1, 2,\ldots\}\Big \},$ and relations
\begin{equation}
R_{12}(\lambda_1, \lambda_2)\ L_1(\lambda_1)\ L_2(\lambda_2) = L_2(\lambda_2)\ 
L_1(\lambda_1)\  R_{12}(\lambda_1,  \lambda_2), \label{RTT}
\end{equation}
where $R_{12} =R \otimes  1_{\mathfrak A }$ and 
$L_{1}=\sum_{x,y\in X }e_{x,y}\otimes 1_V \otimes  L_{x,y}$\footnote{Notice that in $L$ in addition to the indices 1 and 2 in (\ref{RTT}) there is also an implicit ``quantum index'' $3$ associated to ${\mathfrak A},$ 
which for now is omitted, i.e. one writes $L_{13}, L_{23}$.}, $L_{2}=\sum_{x,y \in X} 1_V \otimes  e_{x,y}\otimes  L_{x,y}.$
\end{definition}
We note that if equation (\ref{RTT}) holds, then $R$ is a solution of the Yang-Baxter equation (\ref{frt}) (see, e.g., \cite{Majid} for a proof; see also relevant Remark \ref{rem00}). Definition \ref{frt} states that different choices of solutions of the Yang-Baxter equation yield distinct quantum algebras.

\subsection{The Yangian \texorpdfstring{${\cal Y}(\mathfrak{gl}_n)$}{Y(gl\_n)}}

We give a brief review of a special example of a quantum algebra, the $\mathfrak{gl}_n$ Yangian 
${\cal Y}(\mathfrak{gl}_n)$ (or sometimes ${\cal Y}$ in this manuscript for brevity; 
for a more detailed exposition, the interested reader is referred for instance to \cite{Chari, Molev}). 
We consider the FRT point of view (Definition \ref{frt1}). Specifically, in the case of the Yangian, the $R$-matrix is given by $R(\lambda_1, \lambda_2) =  1_{V \otimes V} +(\lambda_1 - \lambda_2)^{-1}{\cal P},$ where ${\cal P} = \sum_{i,j \in X} e_{i,j} \otimes e_{j,i}$ is the permutation operator, such that ${\cal P}(a\otimes b) = b \otimes a,$ $a,b \in V$ and 
\[L(\lambda) = 1 + \sum_{m=1}^{\infty} \lambda^{-m} L^{(m)}, \, L^{(m)} = \sum_{x,y\in X} e_{x,y} \otimes L_{x,y}^{(m)}.\] Then, by the fundamental relation (\ref{RTT}) the algebraic relations among the generators $L_{x,y}^{(m)}$ of the $\mathfrak{gl}_{n}$ Yangian are deduced and are given in the following definition (the interested reader is referred to \cite{Molev} for a more detailed discussion on Yangians).

\begin{definition} Let $X$ be some finite set of cardinality $n\in {\mathbb Z}^+$. The $\mathfrak{gl}_{n}$ Yangian ${\cal Y}(\mathfrak{gl}_{n})$ (or ${\cal Y}$ for brevity) is a unital, associative algebra generated by indeterminates $1_{\cal Y}$ (unit element) and $L_{i,j}^{(m)},$ $i,j \in X,$ $m \in \{0,1,2, \ldots\}$ ($L_{i,j}^{(0)}= \delta_{i,j}1_{\cal Y}$) and relations:
\begin{equation}
\Big [ L_{i,j}^{(p+1)},\ L_{k,l}^{(m)}\Big ] -\Big [ L_{i,j}^{(p)},\ L_{k,l}^{(m+1)}\Big ] = L_{k,j}^{(m)}L_{i,l}^{(p)}- L_{k,j}^{(p)}L_{i,l}^{(m)},  \label{fund2b}
\end{equation}
where $[ \ \,,\ ]: {\cal Y}(\mathfrak{gl}_n)\times {\cal Y}(\mathfrak{gl}_n)\to {\cal Y}(\mathfrak{gl}_n),$ such that $[a, b] = ab - ba,$ for all $a,b \in {\cal Y}.$ 
\end{definition}
Let us focus on the first few explicit exchange relations from (\ref{fund2b})
\begin{enumerate}
\item  $p=0,$ $m=1$ ($L_{i,j}^{(0)} = \delta_{i,j}$):
\begin{equation}
\big [ L_{i,j}^{(1)},\ L_{k,l}^{(1)}\big ] = \delta_{i,l} L_{k,j}^{(1)} -\delta_{k,j} L_{i,l}^{(1)}  \nonumber
\end{equation}
the latter are the familiar ${\mathfrak gl}_{n}$ exchange  relations.

\item  $p=2,$ $m=0$:
\begin{equation}
\big [L_{i,j}^{(2)},\ L_{k,l}^{(1)}\big ] =\delta_{i,l} L_{k,j}^{(2)} - \delta_{k,j} L_{i,l}^{(2)}  \nonumber
\end{equation}

\item  $p=2,$ $m=1$:
\begin{equation}
\big [ L_{i,j}^{(3)},\ L_{k,l}^{(1)}\big ] -\big [ L_{i,j}^{(2)},\ L_{k,l}^{(2)}\big ] =
L_{k,j}^{(1)} L_{i,l}^{(2)} - L_{k,j}^{(2)} L_{i,l}^{(1)} \nonumber
\end{equation}

\item $p=3,$ $m=0$
\begin{equation}
\big [L_{i,j}^{(3)},\ L_{k,l}^{(1)}\big ] =\delta_{i,l} L_{k,j}^{(3)} - \delta_{k,j} L_{i,l}^{(3)}  \nonumber
\end{equation}
\end{enumerate}

\begin{remark} \label{remy}
The Yangian is a quasi-triangular Hopf algebra on ${\mathbb C}$ \cite{Drinfeld} equipped with (recall Definitions \ref{hopf}, \ref{quasi} and Remark \ref{rem00}):
\begin{enumerate}
 \item A co-product $\Delta: {\cal Y}(\mathfrak {gl}_n) \to {\cal Y}(\mathfrak {gl}_n) \otimes {\cal Y}(\mathfrak {gl}_n)$ 
such that
$(\id \otimes \Delta)L(\lambda) =L_{13}(\lambda) L_{12}(\lambda).$
\item A counit $ \epsilon: {\cal Y}(\mathfrak{gl}_n) \to {\mathbb C},$ such that $(\id \otimes \epsilon) L(\lambda) = 1_{V}.$
\item An antipode $s: {\cal Y}(\mathfrak{gl}_n) \to {\cal Y}(\mathfrak{gl}_n):$ $(\id \otimes s)L^{-1}(\lambda) =L(\lambda).$
\end{enumerate}
We recall that $L(\lambda) = \sum_{m=0}^{\infty} \lambda^{-m } L^{(m)}= \sum_{m=0}^{\infty } \sum_{a,b \in X}\lambda^{-m}e_{a,b} \otimes L_{a,b}^{(m)},$ 
then the coproducts of the Yangian generators $L_{a,b}^{(n)}$ are given as ($L^{(0)}_{a,b} = \delta_{a,b}1_{\cal Y}$)
\begin{equation}
\Delta(L^{(m)}_{a,b}) = \sum_{c\in X}\sum_{k=0}^{m} L^{(k)}_{c,b} \otimes L^{(m-k)}_{a,c}
\end{equation}
For instance, the first couple of generators of the Yangian are given for $a,b \in X,$ as
\begin{eqnarray}
\Delta(L_{a,b}^{(1)}) &=& L_{a,b}^{(1)} \otimes 1_{\cal Y} + 1_{\cal Y} \otimes
L_{a,b}^{(1)} \cr
\Delta(L_{a,b}^{(2)}) &=& L_{a,b}^{(2)}
\otimes 1_{\cal Y}+ 1_{\cal Y}\otimes L_{a,b}^{(2)} +
\sum _{c\in X}^n
L_{c,b}^{(1)}\otimes L_{a,c}^{(1)}, \cr
\Delta(L_{a,b}^{(3)}) &=& L_{a,b}^{(3)}
\otimes 1_{\cal Y}+ 1_{\cal Y}\otimes L_{a,b}^{(3)} +
\sum _{c\in X}
L_{c,b}^{(1)}\otimes L_{a,c}^{(2)}  +
\sum _{c\in X}
L_{c,b}^{(2)}\otimes L_{a,c}^{(1)},\  \ldots \label{cop}
\end{eqnarray}
The Yangian as a Hopf algebra is {co-associative}, and the $n$-coproducts can be derived by iteration via 
$\Delta^{(n+1)} = (\id \otimes \Delta^{(n)})\Delta = (\Delta^{(n)} \otimes \id)\Delta.$

Moreover, the counit exists $\epsilon: {\cal Y}(\mathfrak{gl}_{n}) \to {\mathbb C},$ such that 
\[(\epsilon \otimes \id)\Delta(L_{a,b}^{(m)}) = (\id \otimes \epsilon)\Delta(L^{(m)}_{a,b}) = L_{a,b}^{(m)},\]
and hence we obtain by iteration that $\epsilon(L^{(m)}_{a,b}) =0,$ for all $a,b \in X$  and $m \in{\mathbb Z}^+.$ 
The antipode $s: {\cal Y}(\mathfrak{gl}_{n})\to {\cal Y}(\mathfrak{gl}_{n})$ exists, such that 
\[m(s\otimes \id)\Delta(L_{a,b}^{(m)}) =m(\id \otimes s)\Delta(L_{a,b}^{(m)})=\epsilon(L^{(m)}_{a,b})1_{\cal Y}\]
and recalling that $\epsilon(L_{a,b}^{(m)}) =0,$ we obtain the antipode for each generator via:
\begin{equation}
\sum_{c\in X}\sum_{k=0}^{m} s(L^{(k)}_{c,b})  L^{(m-k)}_{a,c}=\sum_{c\in X}\sum_{k=0}^{m} L^{(k)}_{c,b}  s(L^{(m-k)}_{a,c})=0.
\end{equation}
For example, the antipode for the first couple of generators is given as:
\begin{eqnarray}
&& s(L^{(1)}_{a,b}) = - L_{a,b}^{(1)} \nonumber\\
&& s(L^{(2)}_{a,b}) =
 -L^{(2)}_{a,b} +\sum_{c\in X}L^{(1)}_{c,b}L^{(1)}_{a,c},\nonumber\\
 && s(L^{(3)}_{a,b}) =
 -L^{(3)}_{a,b} +\sum_{c\in X}L^{(1)}_{c,b}L^{(2)}_{a,c} + \sum_{c\in X}L^{(2)}_{c,b}L^{(1)}_{a,c} -\sum_{c,d\in X} L^{(1)}_{d,b}L^{(1)}_{c,d}L^{(1)}_{a,c},\ \ldots \end{eqnarray}
\end{remark}


\subsection{Twisting the Yangian}

 Before we present the main findings regarding the twisting of the $\mathfrak{gl}_n$ Yangian we give the definition of the {\it augmented $\mathfrak{gl}_n$ Yangian}.

\begin{definition} Let $X$ be a finite non-empty set and for all $a\in X,$ $\sigma_a, \tau_a: X \to X,$ such that $\sigma_a$ is bijecitve. The augmented $\mathfrak{gl}_n$ Yangian, denoted as ${\cal Y}_n^+,$ is a unital, associative algebra generated by indeterminates $1_{\cal Y},$ $L_{a,b}^{(m)},$ $w^{\pm 1}_a,$ $h_a,$ $a, b \in X,$ $m \in \{0,1,2, \ldots\}$ ($L_{a,b}^{(0)}= \delta_{a,b}1_{\cal Y}$) and relations
\begin{eqnarray}
&& \Big [ L_{a,b}^{(p+1)},\ L_{c,d}^{(m)}\Big ] -\Big [ L_{a,b}^{(p)},\ L_{c,d}^{(m+1)}\Big ] = L_{c,b}^{(m)}L_{a,d}^{(p)}- L_{c,b}^{(p)}L_{a,d}^{(m)},  \nonumber\\
&& h_a  h_b =\delta_{a,b} h_a, ~~ w_a^{-1} w_a =w_aw_a^{-1} =1_{ {\cal Y}}, ~~  w_a w_b = 
w_{\sigma_a(b)} w_{\tau_{b}(a)} ~~  w_a h_b = h_{\sigma_a(b)} w_a, \nonumber\\ 
&& w_a L^{(p)}_{b,c} = L^{(p)}_{\sigma_a(b), \sigma_a(c)} w_a, ~~~h_b L^{(p)}_{a,b} = L^{(p)}_{a,b} h_a,   ~~~h_c L^{(p)}_{a,b} = L^{(p)}_{a,b} h_c= 0~~ \mbox{if} ~~c \neq  a,b. \label{fund2c}
\end{eqnarray}\end{definition}

\begin{theorem} \label{bob} Let ${\cal Y}_n^+$ be the augmented $\mathfrak{gl}_n$ Yangian and $+: X \times X \to X,$ such that $(a,b) \mapsto a+b$. If $(X, +, 0)$ is a group and for all $a, b, x \in X,$ 
\begin{equation}
 \sigma_x(a) + \sigma_x(b) = \sigma_x(a + b), \label{condition0c}
\end{equation} 
then ${\cal Y}_n^+$ is a Hopf algebra, with co-product $\Delta: {\cal Y}_n^+ \to {\cal Y}_n^+ \otimes {\cal Y}^+_n,$ such that
\[\Delta(w_a^{\pm 1}) = 
w_a^{\pm 1}\otimes w_a^{\pm 1},\, \Delta(h_a) = \sum_{b, c \in X} h_b \otimes h_c \big |_{b+c =a}\] and \[\Delta(L^{(m)}_{a,b}) = \sum_{k=1}^m\sum_{c\in X} L_{c,b}^{(k)} \otimes L^{(m-k)}_{a,c}\] for all $a,b \in X$ and $m\in {\mathbb Z}^+.$
\end{theorem}
\begin{proof}
    The proof is based on the fact that ${\cal Y}(\mathfrak{gl}_n)$ is a Hopf algebra (see Remark \ref{remy}) and on Theorem \ref{basica2b}. 
    Co-associativity holds (Theorem \ref{basica2b}) and it is straightforward to show that $\Delta: {\cal Y}^+_n \to {\cal Y}^+_n  \otimes {\cal Y}^+_n $ is an algebra homomorphism. The counits and antipodes of the algebra generators are uniquely defined from the basic axioms of the Hopf algebra (see also Theorem \ref{basica2b} and Remark \ref{remy}).
\end{proof}

\begin{proposition} \label{pror} 
Consider the representation $\rho: {\cal Y}_n^+ \to \mbox{End}({\mathbb C}^n),$ 
such that for all $a,b \in X,$ $m\in {\mathbb Z}^+$
\begin{equation}
\rho(L^{(m)}_{a,b}) = e_{b,a},~~~\rho(w_a) = \sum_{c\in X}e_{\sigma_{a}(c), c}, ~~~\rho(h_a) = e_{a,a}. \label{frep} \end{equation}
Let also the Yangian $R$-matrix, $R(\lambda) \in End({\mathbb C}^{n} \otimes {\mathbb C}^n),$ $R(\lambda) = 1+{\frac{1}{\lambda}}{\cal P},$ 
where $\lambda \in {\mathbb C},$ ${\cal P} =\sum_{a,b \in X}e_{a,b} \otimes e_{b,a}$ is the permutation operator and $L(\lambda) = 1 + \sum_{m=1}^{\infty}\lambda^{-m}L^{(m)},$ where $L^{(m)} = \sum_{a,b \in X} e_{a,b} \otimes L_{a,b}^{(m)},$ $L_{a,b}^{(m)} \in {\cal Y}(\mathfrak{gl}_n).$ Then,

\begin{enumerate} 
\item For all $f \in {\cal Y}_n^+,$
$$\big((\rho \otimes \rho)\Delta^{(op)}(f)\big) R(\lambda) = R(\lambda) \big((\rho \otimes \rho)\Delta(f)\big),$$  $$\big((\rho \otimes \id)\Delta^{(op)}(f)\big) L(\lambda)=  L(\lambda) \big((\rho \otimes \id)\Delta(f)\big).$$
\item Let also ${\cal F} := \sum_{a\in X}h_a \otimes w_a$ and ${\cal F}_{123}, {\cal F}_{12,3}, {\cal F}_{1,23}$ are defined in Theorem \ref{twist2}. Moreover, $F:= (\rho\otimes \rho){\cal F},$ ${\mathrm F} := (\rho\otimes \id) {\cal F},$  $R^F(\lambda) = F^{(op)}R(\lambda)F^{-1},$ $L^F(\lambda) = {\mathrm F}^{(op)}L(\lambda){\mathrm F}^{-1},$ (${\mathrm F}^{(op)} = (\rho \otimes \id){\cal F}^{(op)}$) and ${\mathrm F}_{123}: = (\rho \otimes \rho \otimes \id){\cal F}_{123},$ then
\begin{equation}
R^F_{12}(\lambda_1 -\lambda_2)L^F_1(\lambda_1) L^F_2(\lambda_2) =  
L^F_2(\lambda_2) L^F_1(\lambda_1)R^F_{12}(\lambda_1 -\lambda_2). \label{this1}
\end{equation}
\end{enumerate}
\end{proposition}
\begin{proof} 
\begin{enumerate}
\item The proof is based on the algebraic relations of ${\cal Y}_n^+$   
and the expressions of the co-products of the algebra generators are given in Theorem \ref{bob}.
\item In the proof of the second part we use part one of the proposition as well as Theorem \ref{twist2}. Specifically, we start from equation (\ref{RTT}) for the Yangian and as in the proof of Theorem \ref{twist2} and using the fundamental representation (\ref{frep}):
\begin{eqnarray}
&& {\mathrm F}_{321} R_{12}(\lambda_1 -\lambda_2) L_{13}(\lambda_1) L_{23}(\lambda_2) ={\mathrm F}_{321} L_{23}(\lambda_2) L_{13}(\lambda_1) R_{12}(\lambda_1 -\lambda_2) \ \Rightarrow\ \nonumber \\ &&  R^F_{12}(\lambda_1 -\lambda_2) L^F_{13}(\lambda_1) L^F_{23}(\lambda_2) {\mathrm F}_{123}=  L^F_{23}(\lambda_2) L^F_{13}(\lambda_1) R^F_{12}(\lambda_1 -\lambda_2) {\mathrm F}_{123}, \nonumber \end{eqnarray} 
which leads to (\ref{this1}), due to that fact that ${\mathrm F}_{123}$ is invertible. \hfill \qedhere
\end{enumerate}
    \end{proof}

\begin{remark}
According to Proposition \ref{pror} $R^{F} = r + {\frac{1}{\lambda}} {\cal P},$ where ${\cal P} = \sum_{a,b \in X} e_{a,b} \otimes e_{b,a} $  and $r = \sum_{a,b \in X} e_{b, \sigma_a(b)} \otimes e_{a, \tau_b(a)}.$ Moreover,
$L^{F}(\lambda) = L^{'(0)} + \sum_{m=1}^{\infty} \lambda^{-m}  L^{'(m)},$ where
\begin{equation}
L^{'(0)} = \sum_{a,b\in X} e_{b, \sigma_a(b)} \otimes h_a w_{\sigma_a(b)}, ~~~~ L^{'(m)} = \sum_{a,b,c \in X} e_{a,b} \otimes h_c L^{(m)}_{\sigma_c(a),b}w_b.  \nonumber
\end{equation}
Moreover, the twisted coproducts for the augmented $\mathfrak{gl}_n$ Yangian are $\Delta_F(x) = {\cal F} \Delta(x) {\cal F}^{-1},$  $x \in {\cal Y}^+_n.$
Recall first the twisted coproducts of the special set-theoretic YB algebra, for $a\in X,$ $(X,+)$ is a group, \[ \Delta_F(w_a) = \sum_{b \in X}  w_ah_b \otimes w _{\tau_b(a)}, \,
\Delta_F(h_a) = \sum_{b \in X} h_b \otimes  h_c \big |_{b + \sigma_b(c) =a}\] (recall Remark \ref{remtw}). Also, for $a, b \in X,$ \[\Delta_F(L^{(m)}_{a,b}) = \sum_{k=1}^m \sum_{c \in X} L^{(k)}_{c,b} h_c \otimes w_b^{-1} L^{(m-k)}_{a,c}w_c.\]
\end{remark}

\subsubsection*{Acknowledgments}
Support from the EPSRC research grant EP/V008129/1 is acknowledged.


\EditInfo{May 29, 2025}{September 8, 2025}{Ivan Kaygorodov and David Towers}

\end{document}